\documentclass[twoside,11pt,reqno]{amsart}
\usepackage{amsmath,amssymb,amscd,mathrsfs,amscd, todonotes}
\usepackage{graphics,verbatim}
\usepackage{todonotes}
\usepackage{enumitem}
\usepackage{placeins}
\usepackage{hyperref}

\newcommand{\losemi}{{\otimes \kern -.78em \ltimes}}
\newcommand{\rosemi}{{\otimes \kern -.78em \rtimes}}

\newcommand{\Hom}{\ensuremath{\operatorname{Hom}}}

\newcommand{\Ext}{\operatorname{Ext}}

\newcommand{\la}{\lambda}

\newcommand{\St}{\operatorname{St}}

\newcommand{\soc}{\operatorname{soc}}
\newcommand{\rad}{\operatorname{rad}}

\makeatletter
\newcommand{\leqnomode}{\tagsleft@true}
\newcommand{\reqnomode}{\tagsleft@false}
\makeatother

\newtheorem{theorem}{Theorem}[subsection]

\makeatletter\let\c@fact\c@theorem\makeatother

\makeatletter\let\c@note\c@theorem\makeatother

\makeatletter\let\c@lemma\c@theorem\makeatother

\makeatletter\let\c@lemma\c@theorem\makeatother

\newtheorem{quest}{Question}[subsection]
\makeatletter\let\c@ques\c@theorem\makeatother
\makeatletter\let\c@ques\c@conjecture\makeatother

\newtheorem{prop}{Proposition}[subsection]
\makeatletter\let\c@prop\c@theorem\makeatother

\newtheorem{conjecture}{Conjecture}[subsection]
\makeatletter\let\c@conjecture\c@theorem\makeatother
\makeatletter\let\c@quest\c@conjecture\makeatother

\makeatletter\let\c@cor\c@theorem\makeatother

\makeatletter\let\c@defn\c@theorem\makeatother

\theoremstyle{definition}

\newtheorem{remark}{Remark}[subsection]
\makeatletter\let\c@remark\c@theorem\makeatother

\makeatletter\let\c@example\c@theorem\makeatother
\numberwithin{equation}{subsection}

\usepackage[capitalise]{cleveref}
\crefname{theorem}{Theorem}{Theorems}
\crefname{fact}{Fact}{Facts}
\crefname{note}{Note}{Notes}
\crefname{lemma}{Lemma}{Lemmas}
\crefname{alg}{Algorithm}{Algorithms}
\crefname{remark}{Remark}{Remarks}
\crefname{example}{Example}{Examples}
\crefname{prop}{Proposition}{Propositions}
\crefname{conj}{Conjecture}{Conjectures}
\crefname{cor}{Corollary}{Corollaries}
\crefname{defn}{Definition}{Definitions}
\crefname{quest}{Question}{Questions}
\crefname{equation}{\!\!}{\!\!} %Remove spacing around phantom equation name

\newcounter{listequation}

\begin{document}

\title{Counterexamples to the Tilting and $(p,r)$-Filtration Conjectures}

\begin{abstract} 
In this paper the authors produce a projective indecomposable module for the Frobenius kernel of a simple algebraic group in characteristic $p$ that is not the restriction of an indecomposable tilting module.  This yields a counterexample to Donkin's longstanding Tilting Module Conjecture.  The authors also produce a Weyl module that does not admit a $p$-Weyl filtration.  This answers an old question of Jantzen, and also provides a counterexample to the $(p,r)$-Filtration Conjecture.
\end{abstract}

\author{\sc Christopher P. Bendel}
\address
{Department of Mathematics, Statistics and Computer Science\\
University of
Wisconsin-Stout \\
Menomonie\\ WI~54751, USA}
\thanks{Research of the first author was supported in part by Simons Foundation Collaboration Grant 317062}
\email{bendelc@uwstout.edu}

\author{\sc Daniel K. Nakano}
\address
{Department of Mathematics\\ University of Georgia \\
Athens\\ GA~30602, USA}
\thanks{Research of the second author was supported in part by
NSF grant DMS-1701768}
\email{nakano@math.uga.edu}

\author{\sc Cornelius Pillen}
\address{Department of Mathematics and Statistics \\ University
of South
Alabama\\
Mobile\\ AL~36688, USA}
\thanks{Research of the third author was supported in part by Simons Foundation Collaboration Grant 245236}
\email{pillen@southalabama.edu}

\author{Paul Sobaje}
\address{Department of Mathematical Sciences \\
          Georgia Southern University\\
          Statesboro, GA~30458, USA}
\email{psobaje@georgiasouthern.edu}
%\thanks{Research of the fourth author was partially supported by NSF RTG grant  DMS-1344994}
\date{\today}
\date\today

\maketitle

\section{Introduction}

\subsection{} Let $G$ be a semisimple, simply connected algebraic group over an algebraically closed field of characteristic $p>0$ and ${\mathfrak g}$ 
be its Lie algebra. Restricted representations for the Lie algebra ${\mathfrak g}$ are equivalent to representations for the first Frobenius 
kernel $G_{1}$. In the 1960s Curtis showed that the simple $G_{1}$-modules lift to simple modules for $G$. Later, 
Humphreys and Verma investigated the projective indecomposable modules for $G_{1}$ and asked whether these modules 
have a compatible $G$-structure. This statement was verified for $p\geq 2h-2$ (where $h$ is the Coxeter number) by work of Ballard \cite{B} and Jantzen \cite{J}. For over 50 years, it has been anticipated that the Humphreys-Verma Conjecture would hold for all $p$. 

In 1990, Donkin presented a series of conjectures at MSRI. One of the conjectures, known as the Tilting Module Conjecture, states that 
a projective indecomposable module for $G_r$ can be realized as an indecomposable tilting $G$-module (see Conjecture~\ref{tilting}). 
Like the Humphreys-Verma Conjecture, the Tilting Module Conjecture holds for $p\geq 2h-2$ with the hope of 
being valid for all $p$. Recently, the Tilting Module Conjecture has been shown to be related to another one of Donkin's conjectures 
involving good $(p,r)$-filtrations. A more detailed exposition with the connections is presented in Section~\ref{S:conjectures}. 

The Tilting Module Conjecture has taken on additional importance following work by Achar, Makisumi, Riche, and Williamson \cite{AMRW}, who have shown that when $p > h$, the characters of indecomposable tilting modules can be given via $p$-Kazhdan-Lusztig polynomials, confirming a conjecture by Riche and Williamson \cite{RW}.  When $p \ge 2h-2$, the Tilting Module Conjecture then allows one to deduce the characters of simple $G$-modules. The authors of \cite{AMRW} credit Andersen with this observation.

\subsection{} The goal of this paper is to present counterexamples to the conjectures and questions stated in Section~\ref{S:conjectures}. 
In this subsection, let $G$ be a simple algebraic group whose root system is of type $G_2$ and $p=2$.
In particular, we 
\begin{itemize} 
\item[(1.2.1)] present a counterexample to the Tilting Module Conjecture - see Theorem \ref{tilt:no}; 
\item[(1.2.2)] construct a counterexample to one direction of Donkin's Good $(p,r)$-Filtration Conjecture (i.e., Conjecture~\ref{donkinconj}($\Leftarrow$)) - see Theorem \ref{T:no2good} and Section \ref{S:moduleM}; 
\item[(1.2.3)] give an example of a costandard/induced module $\nabla(\la)$ that does not admit a good $(p,r)$-filtration - see Theorem \ref{T:no2good}. 
\end{itemize}  

Specifically, we demonstrate that there does not exist a good $2$-filtration for the induced module $\nabla(2,1)$.\footnote{A major step in this process was a computation of a filtration of $\Delta(2,1)$, obtained using Stephen Doty's WeylModule package for the software GAP \cite{Doty,GAP}, that, when dualized, indicated that $\nabla(2,1)$ could not have a good $2$-filtration.} This gives a negative answer to an open question of Jantzen \cite{J}, and this module is also is a counterexample for (1.2.2). As a consequence of these results, we prove that the indecomposable tilting module $T(2,2)$ is decomposable over the first Frobenius kernel of $G$. We present a formal proof of this fact using information about extensions of simple $G$-modules of small highest weights. 
\footnote{This fact was verified in another way by running Doty's GAP program to compute that the socle of $\Delta(2,2)$ is isomorphic to $k \oplus L(0,1)$.  As $\Delta(2,2)$ is a submodule of $T(2,2)$, one concludes that the socle of $T(2,2)$ has at least two factors over $G_1$, so that $T(2,2)$ splits into at least two projective summands over $G_1$.}  

%Give a counter-example to one direction of good (p,r)-filtration conjecture (and as Cornelius pointed out, we can take the top two laters of \nabla(2,1) as a subcounter-example, this is a module having two composition factors so basically a minimal counter-example (unless other direction of (p,r)-conjecture fails somewhere.

\subsection{Acknowledgements} The authors would like to thank Henning H. Andersen and Jens C. Jantzen for useful comments and suggestions on an earlier version of this manuscript.

\section{Preliminaries}

\subsection{Notation.} The notation will follow the conventions in \cite[Section 2.1]{BNPS}, most of which follow those in \cite{rags} (though our notation for induced and Weyl modules follows the costandard and standard module conventions in highest weight category literature). Let $G$ be a connected, 
semisimple algebraic group scheme defined over ${\mathbb F}_{p}$ and $G_{r}$ be its $r$th Frobenius kernel. 

Let $X_{+}$ denote the dominant weights for $G$, and $X_{r}$ be the $p^{r}$-restricted weights. For $\lambda\in X_{+}$, there are four fundamental classes of $G$-modules (each having highest weight $\lambda$): 
$L(\lambda)$ (simple), $\nabla(\lambda)$ (costandard/induced), $\Delta(\lambda)$ (standard/Weyl), and $T(\lambda)$ (indecomposable tilting). A $G$-module $M$ has a {\em good filtration} 
(resp. {\em Weyl filtration}) if and only if $M$ has a filtration with factors of the form $\nabla(\mu)$ (resp. $\Delta(\mu)$) for suitable $\mu\in X_+$.  

For $\lambda\in X_+$ with unique decomposition $\lambda = \lambda_0 + p^r\lambda_1$ with $\lambda_0\in X_r$ and $\lambda_1\in X_+$, define $\nabla^{(p,r)}(\lambda) = L(\lambda_0)\otimes \nabla(\lambda_1)^{(r)}$ where $(r)$ denotes the twisting of the module action by the $r$th Frobenius morphism. Similarly, 
set $\Delta^{(p,r)}(\lambda) = L(\lambda_0)\otimes \Delta(\lambda_1)^{(r)}$.  A $G$-module $M$ has a {\em good $(p,r)$-filtration} 
(resp. {\em Weyl $(p,r)$-filtration}) if and only if $M$ has a filtration with factors of the form $\nabla^{(p,r)}(\mu)$ (resp. $\Delta^{(p,r)}(\mu)$) for suitable $\mu\in X_+$. In the case when $r=1$, we often refer to 
good $(p,1)$-filtrations as good $p$-filtrations. 

Let $\rho$ be the sum of the fundamental weights and $\text{St}_r = L((p^r-1)\rho)$ (which is also isomorphic to $\nabla((p^r-1)\rho)$ and $\Delta((p^r-1)\rho)$) be the $r$th Steinberg module. For $\lambda\in X_{r}$, let $Q_{r}(\lambda)$ denote the projective cover (equivalently, injective hull) of $L(\lambda)$ as a $G_{r}$-module. If $\lambda\in X_{r}$, set $\hat{\lambda}=2(p^{r}-1)\rho+w_{0}\lambda$ 
where $w_{0}$ is the long element in the Weyl group $W$. 

Let $M$ be a finite-dimensional $G$-module, and let 
$$M\supseteq \text{rad}_{G} M \supseteq \text{rad}^{2}_{G} M \supseteq \dots \supseteq \{0\}$$ 
be the radical series of $M$. Moreover, let 
$$\{0\} \subseteq \text{soc}_{G} M \subseteq \text{soc}^{2}_{G} M \subseteq \dots \subseteq M$$ 
be the socle series for $M$. One can similarly define such filtrations for $G_{r}$-modules. 

\subsection{The Conjectures.} \label{S:conjectures}  In the early 1970s Humphreys and Verma presented the following conjecture on the lifting of $G$-structures on the projective modules for $G_{r}$. 

\begin{conjecture} \label{lifting} For $\lambda\in X_{r}$, the $G_{r}$-module structure on $Q_{r}(\lambda)$ can be lifted to $G$. 
\end{conjecture} 

The conjecture was first verified by Ballard for $p\geq 3h-3$ \cite{B} and then by Jantzen for $p\geq 2h-2$ \cite{J}, who further showed under this improved bound that the $G$-structure was unique up to isomorphism. Later, at a conference at MSRI in 1990, Donkin presented the following conjecture, predicting that a $G$-module structure on $Q_{r}(\lambda)$ arises from a specific tilting module which must be \textit{the} $G$-module structure whenever uniqueness of $G$-structure holds. 
 
\begin{conjecture}\label{tilting}  For all $\lambda\in X_{r}$, $T(2(p^{r}-1)\rho+w_{0}\lambda)|_{G_{r}}=Q_{r}(\lambda)$. 
\end{conjecture} 

Conjecture~\ref{tilting} holds for $p\geq 2h-2$ and the proof under this bound entails locating one particular $G$-summand of $\text{St}_{r}\otimes L(\lambda)$. At the same conference at MSRI, another conjecture was introduced by Donkin that interrelates good filtrations with good $(p,r)$-filtrations via the Steinberg module. 

\begin{conjecture} \label{donkinconj} Let $M$ be a finite-dimensional $G$-module. Then $M$ has a good $(p,r)$-filtration if and only if $\operatorname{St}_r\otimes M$ has a good filtration. 
\end{conjecture}

We denote the two directions of the statement as follows:
\begin{itemize}
\item Conjecture~\ref{donkinconj}($\Rightarrow$): If $M$ has a good  $(p,r)$-filtration, then $\operatorname{St}_r\otimes M$ has a good filtration.
\item Conjecture~\ref{donkinconj}($\Leftarrow$): If $\operatorname{St}_r\otimes M$ has a good filtration, then $M$ has a good $(p,r)$-filtration.
\end{itemize}

Conjecture~\ref{donkinconj}($\Rightarrow$) is equivalent to $\text{St}_{r}\otimes L(\lambda)$ being a tilting module for all $\lambda\in X_{r}$. Andersen \cite{And} and later Kildetoft and Nakano \cite{KN} verified Conjecture~\ref{donkinconj}($\Rightarrow$) when $p\geq 2h-2$. In a recent paper, the authors lowered the bound to $p \geq 2h-4$ (cf. \cite{BNPS}). For rank 2 groups (including $G_{2}$), Conjecture~\ref{donkinconj}($\Rightarrow$) was proved for all $p$ in \cite{KN} and \cite{BNPS}. 

There are also strong relationships, established by Kildetoft and Nakano \cite{KN} and also by Sobaje \cite{So}, between these conjecture given by the
following hierarchy of implications: 
\begin{center}
\text{Conjecture}~\ref{donkinconj} \ \ $\Rightarrow$ \ \  \text{Conjecture}~\ref{tilting} \ \ $\Rightarrow$ \ \ \text{Conjecture}~\ref{donkinconj}($\Rightarrow$).
\end{center}
While we will provide counterexamples to Conjecture~\ref{tilting} and the full Conjecture~\ref{donkinconj}, we remark that Conjecture~\ref{donkinconj}($\Rightarrow$) may still hold for all $p$. A special case of Conjecture~\ref{donkinconj}($\Leftarrow$) was earlier posed by Jantzen \cite{J}. 

\begin{quest}\label{Jantzen-nabla} For $\lambda\in X_{+}$, does $\nabla(\lambda)$ admit a good $(p,r)$-filtration? 
\end{quest} 

Parshall and Scott affirmatively answered the aforementioned question if $p \ge 2h-2$ and the Lusztig Conjecture holds for the given prime and group \cite{PS}.   Recently, Andersen \cite{And2} has shown this for $p \geq (h-2)h$.

\section{Weyl modules and good $(p,r)$-filtrations for $G_{2}$} \label{S:Weylandgood} 

\subsection{Simple and Projective Modules} Assume throughout this section (and most of the remainder of the paper) that the root system of $G$ is of type $G_2$ and that the prime $p=2$.  We follow the Bourbaki ordering of the simple roots: $\alpha_1$ is the short root and $\alpha_2$ is the long root.  For $a,b \in \mathbb{Z}$, we denote by $(a,b)$ the weight $a\varpi_1+b\varpi_2$, where $\varpi_1$ and $\varpi_2$ are the fundamental dominant weights.  The set of restricted weights is
$$X_1 = \{(0,0), (1,0), (0,1), (1,1)\}.$$
Let $\St=\text{St}_{1}$ denote the first Steinberg module $L(1,1)$. The module $L(0,1) \cong \nabla(0,1) \cong \Delta(0,1)$ is the $14$-dimensional adjoint representation.  Among the four costandard $G$-modules of restricted highest weight, only $\nabla(1,0)$ is not simple, and we have that $\nabla(1,0)/L(1,0) \cong k$.  Every simple $G$-module is self-dual, and the weight lattice and root lattice coincide.

Since the characters of the simple $G$-modules of restricted highest weight are known here, it is possible to compute directly the dimensions of the projective indecomposable $G_1$-modules.  We recall in Table~\ref{table:1} some of the information provided by Humphreys in \cite[18.4, Table 4]{Hu}, originally due to Mertens \cite{M}.

\medskip
\FloatBarrier
\begin{table}[h]
\begin{tabular}{|c|c|c|}
\hline
$\la$ & $\dim L(\la)$ & $\dim Q_1(\la)$\\
\hline
$(0,0)$ & $1$ & $36\cdot 64$\\
$(1,0)$ & $6$ & $12\cdot 64$\\
$(0,1)$ & $14$ & $6\cdot 64$\\
$(1,1)$ & $64$ & $64$\\
\hline
\end{tabular}
\caption{Dimensions of simple and projective $G_1$-modules}
\label{table:1}
\end{table}
\FloatBarrier

\subsection{$\text{Ext}^{1}$-calculations} \label{S:Ext1} In our analysis of the structure of the Weyl modules we will need the following $\text{Ext}^{1}$-calculations that appear in Dowd and Sin \cite[Lemma 3.3]{DS}, part (c) of which dates back to work of Jantzen \cite{J91}. 

\begin{prop} \label{DS-Ext}
One has the following isomorphisms as $G$-modules:  
\begin{itemize}
\item[(a)] $\Ext_{G_1}^1(L(1,0),L(0,1))= 0$ 
\item[(b)] $\Ext_{G_1}^1(L(0,1),L(0,1)) = 0 $
\item[(c)] $\Ext_{G_1}^1(k,L(0,1)) \cong \nabla(1,0)^{(1)}$.
\end{itemize}
\end{prop}

\subsection{Decomposition of $\St\otimes L(\lambda)$, $\lambda\in X_{1}$} Recall that $\St$ is projective over the first Frobenius kernel $G_1$.  Hence, for $\la \in X_1$, $\St\otimes L(\la)$ is also projective over $G_1$.  As the highest weight of $\St\otimes L(\la)$ is $\rho + \la = 2\rho - (\rho - \la)$, which is the same as that of $Q_1(\rho - \la)$, the module $Q_1(\rho - \la)$ is necessarily a $G_1$-summand of $\St\otimes L(\la)$.  The following proposition gives a precise decomposition of $\St\otimes L(\la)$ for each $\la \in X_1$.   

\begin{prop}\label{St:tensor}
We have the following decompositions into projective indecomposable modules over $G_1$:
\begin{itemize}
\item[(a)] $\St \otimes k \cong \St$ 
\item[(b)] $\St \otimes L(1,0) \cong Q_1(0,1)$ 
\item[(c)] $\St \otimes L(0,1)  \cong Q_1(1,0) \oplus \St^{\oplus 2}$ 
\item[(d)] $\St \otimes \St  \cong Q_1(0,0) \oplus Q_1(0,1)^{\oplus 2} \oplus \St^{\oplus 16}$.  
\end{itemize}
\end{prop}

\begin{proof} The first isomorphism is immediate, and the second follows by the module dimensions given in Table~\ref{table:1}.  To get the other two, we use the fact that for any $G$-module $M$,
$$\Hom_{G_1}(\St, \St \otimes M) \cong \Hom_{G_1}(\St \otimes \St, M) \cong M^{T_1},$$
where $T_1$ is the Frobenius kernel of the maximal torus $T$.
Now the weight $0$ appears twice in $L(0,1)$, so that $\St^{\oplus 2} \subseteq \St \otimes L(0,1)$.  There is also an embedding of $L(1,0)$ into $\St \otimes L(0,1)$.  The dimensions in Table~\ref{table:1} then imply that (c) holds.

Finally, the $G_1$-socle of $\St \otimes \St$ is determined by all $L(\la)^{T_1}$ for $\la \in X_1$.  Using a table of weights for $G$-modules (see for example \cite{L}) and the fact that $\St \otimes \St$ is a tilting module, one finds that
$$\soc_{G_1} (\St \otimes \St) \cong k \oplus L(0,1)^{\oplus 2} \oplus  (\St \otimes T(1,0)^{(1)})^{\oplus 2},$$ when viewed as a $G$-module. Note that $\St \otimes T(1,0)^{(1)} \cong \St^{\oplus 8}$ as a $G_1$-module, 
proving (d).
\end{proof}

For $\lambda\in X_{1}$, we know that $\St \otimes L(\lambda)$ is a tilting module \cite{KN} of highest weight $\rho + \la$.   Hence, the indecomposable tilting module $T(\rho + \la)$ embeds in $\St\otimes L(\lambda)$.  Furthermore, the $G_1$-Steinberg block component of any $G$-module splits off as a summand over $G$.  Thus we conclude from Proposition~\ref{St:tensor}:

\begin{theorem}\label{T:Q}
Over $G_1$ there are isomorphisms
\begin{itemize}
\item[(a)] $T(1,1)\cong \St$ 
\item[(b)] $T(2,1) \cong Q_1(0,1)$ 
\item[(c)] $T(1,2) \cong Q_1(1,0)$.
\end{itemize}
\end{theorem}

One can show that these are the unique $G$-structures on these modules, by showing that any $G$-structure on $Q_1(1,0)$ or on $Q_1(0,1)$ must admit a good filtration 
%To do this, one argues that if $\mu < (1,2)$ under the usual partial ordering of weights, then $\Delta(\mu)$ has a $2$-Weyl filtration with respect to those composition factors whose $2$-adic weight decomposition begins with $(1,0)$ or $(0,1)$.  This is sufficient to prove the necessary Ext-vanishing condition for a module to have a good filtration 
(a more detailed explanation of this will be provided in a forthcoming paper).

\subsection{} There exists a surjective homomorphism of $G$-modules
$$T(2,1) \twoheadrightarrow \nabla(2,1).$$
Since $T(2,1) \cong Q_1(0,1)$, $L(0,1)$ is its unique semisimple quotient over $G_1$, and therefore the same holds over $G$ since every simple $G$-module is semisimple over $G_1$.  These facts are then true of its homomorphic image $\nabla(2,1)$.  That is,
$$\textup{rad}_{G_1} \nabla(2,1) = \textup{rad}_G \nabla(2,1)$$
and
$$\nabla(2,1)/\textup{rad}_G \nabla(2,1) \cong L(0,1).$$
Since $T(2,1) \cong Q(0,1)$ as a $G_1$-module, the $G_1$-socle of $T(2,1)$ is $L(0,1)$. 

We now want to compute the second layer of the radical series of $\nabla(2,1)$. This will be accomplished by 
calculating the second socle layer of $T(2,1)$ using the $\Ext^1$-results of Proposition \ref{DS-Ext}.

\begin{prop}\label{P:socle-radical} There exist the following isomorphisms of $G$-modules: 
\begin{itemize} 
\item[(a)] $\soc_{G_1}^2 T(2,1)/ \soc_{G_1} T(2,1) \cong \nabla(1,0)^{(1)}$
\item[(b)] $\soc_{G}^2 T(2,1)/\soc_{G}T(2,1)\cong L(1,0)^{(1)}$
\item[(c)] $\textup{rad}_G \nabla(2,1)/\textup{rad}_G^2 \nabla(2,1) \cong L(1,0)^{(1)}$. 
\end{itemize} 
\end{prop}

\begin{proof} (a) and (b): For $\la\in X_{1}$, one has isomorphisms
\begin{align*}
\Hom_{G_1}(L(\la),T(2,1)/L(0,1)) & \cong \Hom_{G_1}(L(\la),Q_1(0,1)/L(0,1))\\
& \cong \Ext^1_{G_1}(L(\la),L(0,1)),\\
\end{align*}
where the first isomorphism holds since $T(2,1) \cong Q_1(0,1)$, and the second comes from degree shifting in cohomology. Proposition~\ref{DS-Ext} then establishes that
$$\soc_{G_1}^2 T(2,1)/\soc_{G_1} T(2,1)$$
is $7$-dimensional and is trivial as a $G_1$-module.  Considering this, as a $G$-module, its only possible composition factors are $k$ and $L(1,0)^{(1)}$.  Since $k$ does not extend $L(0,1)$ nontrivially over $G$, we conclude that
$$\soc_{G}^2 T(2,1)/\soc_{G} T(2,1) \cong L(1,0)^{(1)},$$
and that
$$\soc_{G_1}^2 T(2,1)/\soc_{G_1} T(2,1) \cong \nabla(1,0)^{(1)}$$
(which agrees with the $G$-module structure in Proposition~\ref{DS-Ext}; this extended argument is included to be precise on the inference of $G$-module structure).

(c): Every tilting $G$-module and every simple $G$-module is self-dual, and $\Delta(2,1)^* \cong \nabla(2,1)$, so we will work in the dual situation.  We have that $\Delta(2,1) \subseteq T(2,1)$, therefore
$$\soc_G^2 \Delta(2,1)/\soc_G  \Delta(2,1) \subseteq \soc_G^2 T(2,1)/\soc_G T(2,1) \cong L(1,0)^{(1)}.$$
But, $\soc_G^2 \Delta(2,1)/\soc_{G} \Delta(2,1)\ne 0$, therefore $\soc_G^2 \Delta(2,1)/\soc_{G} \Delta(2,1) \cong L(1,0)^{(1)}$.

Finally, one has 
$$\textup{rad}_G \nabla(2,1)/\textup{rad}_G^2 \nabla(2,1) \cong (\soc_G^2 \Delta(2,1)/\soc_{G} \Delta(2,1))^* \cong L(1,0)^{(1)}.$$
\end{proof}

\subsection{} This following example answers Question \ref{Jantzen-nabla} in the negative, and it is also a counterexample to Conjecture \ref{donkinconj}($\Leftarrow$), since $\St\otimes\nabla(2,1)$ has a good filtration. 

\begin{theorem}\label{T:no2good}
The module $\nabla(2,1)$ for the group of type $G_2$ does not have a good $2$-filtration.
\end{theorem}

\begin{proof}
Suppose that
$$0 = F_0 \subseteq F_1 \subseteq \cdots \subseteq F_n = \nabla(2,1)$$
is a good $2$-filtration.  In view of the structure of the radical series of $\nabla(2,1)$,  
$$F_n/F_{n-1} \cong L(0,1) \quad \text{and} \quad F_{n-1}/F_{n-2} \cong \nabla(\mu)^{(1)},$$
with $L(1,0)$ being the $G$-head of $\nabla(\mu)$.  Since $2\mu \le (2,1)$ under the usual partial ordering of weights, we have
$$2\langle \mu, \alpha_0^{\vee} \rangle \le \langle (2,1), \alpha_0^{\vee} \rangle = 7,$$
where $\alpha_0$ denotes the maximal short root.
Therefore, 
$$\langle \mu, \alpha_0^{\vee} \rangle \le 3,$$
implying that $\mu \in \{ (0,0), (1,0), (0,1) \}$.  But $L(1,0)$ is not in the head of $\nabla(\mu)$ for any of these choices of $\mu$, therefore no such filtration on $\nabla(2,1)$ is possible.
\end{proof}

\begin{remark}
H.H. Andersen has pointed out to us that the module $\nabla(0,2)$ is uniserial, and that its top two layers are the same as those of $\nabla(2,1)$, so that this module also fails to have a good $2$-filtration.
\end{remark}

\subsection{}\label{S:moduleM} The lack of a good $2$-filtration leads to other interesting phenomena which will factor into our proof that the Tilting Module Conjecture does not hold.

\begin{prop}\label{P:nogood}
For the group $G_2$ with $p = 2$, the module $\St \otimes \textup{rad}_G \nabla(2,1)$ does not have a good filtration.
\end{prop}

\begin{proof}
It suffices to show that the Steinberg block component of this module does not admit a good filtration.  Any composition factor of $\St \otimes \textup{rad}_G \nabla(2,1)$ that lies within the Steinberg block has the form $\St \otimes L(\mu)^{(1)}$.  Further, for any such composition factor, we have
$2\mu \le (2,1)$, and as in the previous proof one has $\mu \in \{(0,0), (1,0), (0,1)\}$.  Since $L(1,0)^{(1)}$ is the head of $\textup{rad}_G \nabla(2,1)$, $\St \otimes L(1,0)^{(1)}$ must appear in the head of (the Steinberg block of) $\St \otimes \textup{rad}_G \nabla(2,1)$. 
But we again reason as in the proof above.  If the Steinberg block of $\St \otimes \textup{rad}_G \nabla(2,1)$ has a good filtration, then there is some $\nabla(\mu)$ such that $L(1,0)$ is the head of $\nabla(\mu)$ and $\St \otimes \nabla(\mu)^{(1)}$ is a subquotient of $\St  \otimes \textup{rad}_G \nabla(2,1)$.  But no such subquotient is possible with the limitations on $\mu$.
\end{proof}

\subsection{Conjecture~\ref{donkinconj}($\Leftarrow$): Minimal Counterexample} The module $\St \otimes \nabla(2,1)$ has a good filtration, and none of its $\nabla$-quotients map onto $L(3,1) \cong \St\otimes L(1,0)^{(1)}$.  
It was observed earlier that two copies of $\St$ are contained in $\St \otimes L(0,1)$.  Therefore, it follows that one of these copies nontrivially extends the composition factor $\St \otimes L(1,0)^{(1)}$ in $\St \otimes \textup{rad}_G \nabla(2,1)$ that comes from
$$\St \otimes [\rad_G \nabla(2,1)/\rad_G^2\nabla(2,1)].$$
Now define the $G$-module $M$ via the short exact sequence
\begin{equation}\label{M}
0 \to \rad_G^2T(2,1) \to T(2,1) \to M \to 0.
\end{equation}
Then the non-split  sequences
$$0 \to \rad_G^2\nabla(2,1) \to \nabla(2,1) \to M \to 0$$
and
$$0 \to L(1,0)^{(1)} \to M \to L(0,1) \to 0$$
are immediate consequences of Proposition~\ref{P:socle-radical}.
%Moreover, one also obtains that $\Hom_G(T(2,1), M) =k.$

From weight considerations and Theorem~\ref{T:Q}, it follows that $\St \otimes M \cong T(1,2) \oplus S$, where $S$ is the summand containing all composition factors in the $G_1$-Steinberg block of $\St \otimes M$. We know that $S$ contains $\St \otimes L(1,0)^{(1)}$ once as a composition factor and the Steinberg module twice. No other composition factors occur, and as a consequence of previous discussion, one of the Steinberg factors must sit on top of $\St \otimes L(1,0)^{(1)}$.  In conclusion,
\begin{align*}
\St \otimes M & \cong T(1,2) \oplus (\St \otimes \nabla(1,0)^{(1)}) \oplus \St \\
&  \cong T(1,2) \oplus \nabla(3,1) \oplus \St,
\end{align*}
which has a good filtration.  This then proves the following:

\begin{prop}\label{M:good} Let $M$ be the module defined in \eqref{M}.
\begin{itemize}
\item[(a)] $\St \otimes M$ has a good filtration.
\item[(b)] $\Hom_G( \St, \St \otimes M) =k.$
\end{itemize}
\end{prop}

The module $M$ has composition factors $L(0,1)$ and $L(1,0)^{(1)}$.  Since $L(1,0)^{(1)} \not\cong \nabla(1,0)^{(1)}$, we see that $M$ does not have a good $2$-filtration, even though $\St \otimes M$ has a good filtration.  One could then consider $M$ as a minimal counterexample to Conjecture~\ref{donkinconj}($\Leftarrow$), as it has only two composition factors.  

Indeed, in the general context of a semisimple $G$ and arbitrary prime $p$,  a counterexample with only one composition factor is not possible.  For example, if for some $\la=\la_0+p\la_1$, with $\la_0 \in X_1$ and $\la_1 \in X_+$, the module
$$\St \otimes L(\la_0) \otimes L(\la_1)^{(1)}$$
has a good filtration, then it must be tilting.  But then
$$\St \otimes L(\la_0) \otimes T((p-1)\rho-\la_0) \otimes L(\la_1)^{(1)}$$
is tilting, and since $\St$ is a summand of $L(\la_0) \otimes T((p-1)\rho-\la_0)$, we have that $\St \otimes \St \otimes L(\la_1)^{(1)}$ is also tilting, and then that $\St^{\otimes 3} \otimes L(\la_1)^{(1)}$ is tilting.  But $\St$ is a summand of $\St^{\otimes 3}$, so that $\St \otimes L(\la_1)^{(1)}$ is tilting, and we conclude that $L(\la_1) \cong \nabla(\la_1) \cong T(\la_1)$.  Consequently,  $L(\la_0) \otimes L(\la_1)^{(1)}$ is a good $p$-filtration module.

\section{On The Tilting Module Conjecture}

\subsection{} We return to the assumption that $G$ has a root system of type $G_2$ and the prime $p = 2$. The fact that $\St \otimes \rad_G \nabla(2,1)$ does not have a good filtration guarantees that the Tilting Module Conjecture does not hold in this case.  This essentially follows from \cite[Theorem 5.1.1]{So}, but here we will give a simple self-contained proof of this fact using the results already established in this paper.

\begin{theorem}\label{tilt:no}
The Tilting Module Conjecture does not hold for $G_2$ and $p=2$.
\end{theorem}

\begin{proof}Assume that the Tilting Module Conjecture holds, so that $T(2,2)|_{G_1} \cong Q_1(0,0)$. From the $G$-module structure of the $G_1$-socle of $\St \otimes \St$, as  observed in the proof of Proposition~\ref{St:tensor} part (d), and Theorem~\ref{T:Q}, one then concludes that (as $G$-modules) 
\begin{equation}\label{E:StSt}
\St\otimes\St \cong T(2,2)\oplus T(2,1)^{\oplus 2}\oplus T(3,1)^{\oplus 2}.
\end{equation}
In particular, the tilting module $T(2,1)$ appears twice in the tensor product $\St \otimes \St$.  Let $M$ be the quotient of $T(2,1)$ from Proposition~\ref{M:good}. Then we have that 
$$2 \leq \dim \Hom_G(\St \otimes \St , M) = \dim \Hom_G(\St, M \otimes \St ),$$
a contradiction to part (b) of Proposition~\ref{M:good}.
\end{proof}

\subsection{The socle of $T(2,2)$} There are two copies of $L(0,1)$ in the $G$-socle of $\St \otimes \St$, but we have now established that $T(2,1)$ occurs as a summand of $\St \otimes \St$ at most once  (i.e., the decomposition in \eqref{E:StSt} {\it fails} to hold).  Looking again at Theorem~\ref{T:Q}, it follows that $L(0,1)$ must appear as a submodule of $T(2,2)$.
This fact has been independently confirmed by Doty's program \cite{Doty, GAP}, which has computed more precisely that
$$k \oplus L(0,1) \cong \soc_G \Delta(2,2) \subseteq T(2,2).$$
We note that, whenever $T(\hat{\lambda})=Q_{1}(\lambda)$ as a $G_{1}$-module for $\lambda\in X_{1}$, then $\soc_G \Delta(\hat{\lambda})$ must  be simple and isomorphic to $L(\lambda)$.

\subsection{The Humphreys-Verma Conjecture} Although $T(2,2)$ is not a lift of $Q_1(0,0)$, it is still possible that $Q_1(0,0)$ has some other $G$-module structure, so the Humphreys-Verma Conjecture remains open for now.  Nevertheless, it is significant that even if there is some $G$-structure, it will not occur as a $G$-submodule of $\St \otimes \St$ (though it could appear as a subquotient).  This defies the long held expectation, going back to early work by Humphreys and Verma, that a $G$-structure should occur in precisely this way.

\providecommand{\bysame}{\leavevmode\hbox
to3em{\hrulefill}\thinspace}

\end{document}